\documentclass[oneside]{amsart}
\usepackage[T1]{fontenc}
\usepackage[latin9]{inputenc}
\usepackage{geometry}
\usepackage{mathtools}
\usepackage{amstext}
\usepackage{amsthm}
\usepackage{amssymb}

\makeatletter
\numberwithin{equation}{section}
\numberwithin{figure}{section}
\theoremstyle{plain}
\newtheorem{thm}{\protect\theoremname}
\theoremstyle{plain}
\newtheorem{cor}[thm]{\protect\corollaryname}
\theoremstyle{remark}
\newtheorem{rem}[thm]{\protect\remarkname}
\theoremstyle{plain}
\newtheorem{lem}[thm]{\protect\lemmaname}

\makeatother

\providecommand{\corollaryname}{Corollary}
\providecommand{\lemmaname}{Lemma}
\providecommand{\remarkname}{Remark}
\providecommand{\theoremname}{Theorem}

\begin{document}
\author{Minoru Hirose} \address[Minoru Hirose]{Institute for Advanced Research, Nagoya University,  Furo-cho, Chikusa-ku, Nagoya, 464-8602, Japan} \email{minoru.hirose@math.nagoya-u.ac.jp}
\author{Hideki Murahara} \address[Hideki Murahara]{The University of Kitakyushu,  4-2-1 Kitagata, Kokuraminami-ku, Kitakyushu, Fukuoka, 802-8577, Japan} \email{hmurahara@mathformula.page}
\author{Tomokazu Onozuka} \address[Tomokazu Onozuka]{Institute of Mathematics for Industry, Kyushu University 744, Motooka, Nishi-ku, Fukuoka, 819-0395, Japan} \email{t-onozuka@imi.kyushu-u.ac.jp}
\title{On the asymptotic behavior of the double zeta function for large negative
indices}
\begin{abstract}
In this paper, we investigate an asymptotic behavior of the double
zeta function of Euler-Zagier type for indices with large negative
real parts.
\end{abstract}

\maketitle

\section{Introduction}

The double zeta function is a meromorphic function of two variables
defined by the analytic continuation of
\[
\zeta(s_{1},s_{2})\coloneqq\sum_{0<m_{1}<m_{2}}\frac{1}{m_{1}^{s_{1}}m_{2}^{s_{2}}}
\]
(\cite{Zhao_anacon}, \cite{AET_anacon}). The purpose of this paper
is to investigate the asymptotic behavior of $\zeta(s_{1},s_{2})$
when the real parts of $s_{1}$ and $s_{2}$ are large negative. Recall
that $\frac{\zeta(s)}{f(s)}$ with $f(s)\coloneqq2^{s}\pi^{s-1}\Gamma(1-s)$
is approximated by $\sin(\pi s/2)$ when $s\to-\infty$ since $\frac{\zeta(s)}{f(s)}=\sin(\pi s/2)\zeta(1-s)$.
In this paper, we show that $\frac{\zeta(s_{1},s_{2})}{f(s_{1}+s_{2})}$
is approximated by
\[
-\frac{1}{2}\sin\left(\frac{\pi}{2}(s_{1}+s_{2})\right)+\frac{1}{2}\cos\left(\frac{\pi}{2}(s_{1}+s_{2})\right)\cot\left(\frac{\pi s_{2}}{s_{1}+s_{2}}\right).
\]
More precisely, we prove the following claim. Let $(s)_{j}$ denotes
the Pochhammer symbol, i.e., $(s)_{j}=s(s+1)\cdots(s+j-1)$, $\cot^{(j)}$
denotes the $j$-th derivative of the cotangent function, and ${\rm Coeff}(p(x),x^{j})$
denotes the coefficient of $x^{j}$ in the Taylor expansion of $p(x)$
at $x=0$.
\begin{thm}
\label{thm:intro_main}Fix a real number $\epsilon>0$ and a nonnegative
integer $N$. Let $s_{1}$ and $s_{2}$ be complex variables satisfying
$\left|\Im s_{1}\right|<\frac{1}{\epsilon}$, $\left|\Im s_{2}\right|<\frac{1}{\epsilon}$,
$\epsilon<\frac{\Re s_{1}}{\Re(s_{1}+s_{2})}<1-\epsilon$, and $\left|s_{1}+s_{2}-2k\right|>|s_{1}+s_{2}|^{-1/\epsilon}$
for all $k\in\mathbb{Z}$. Then, as $\Re(s_{1}),\Re(s_{2})\to-\infty$,
we have
\[
\frac{\zeta(s_{1},s_{2})}{f(s_{1}+s_{2})}=-\frac{1}{2}\sin\left(\frac{\pi}{2}(s_{1}+s_{2})\right)+\frac{1}{2}\cos\left(\frac{\pi}{2}(s_{1}+s_{2})\right)\sum_{j=0}^{2N}\frac{\pi^{j}\cot^{(j)}(\pi r_{2})}{(s_{1}+s_{2})_{j}}c_{j}+O\left(\left|s_{1}+s_{2}\right|^{-N-1}\right),
\]
where $r_{j}\coloneqq\frac{s_{j}}{s_{1}+s_{2}}\quad(j=1,2)$ and
\begin{align*}
c_{j} & \coloneqq{\rm Coeff}((1+xr_{2})^{-s_{1}}(1-xr_{1})^{-s_{2}},x^{j})\\
 & =\sum_{2l_{2}+3l_{3}+\cdots+jl_{j}=j}(s_{1}+s_{2})^{l_{2}+\cdots+l_{j}}\prod_{k=2}^{j}\frac{1}{l_{k}!}\left(\frac{r_{1}(-r_{2})^{k}+r_{2}r_{1}^{k}}{k}\right)^{l_{k}}\qquad(j\geq0).
\end{align*}
\end{thm}

\begin{cor}[The case $N=0$]
Let $s_{1}$ and $s_{2}$ be complex variables satisfying the same
conditions in Theorem \ref{thm:intro_main}. Then, as $\Re(s_{1}),\Re(s_{2})\to-\infty$,
we have
\[
\frac{\zeta(s_{1},s_{2})}{f(s_{1}+s_{2})}=-\frac{1}{2}\sin\left(\frac{\pi}{2}(s_{1}+s_{2})\right)+\frac{1}{2}\cos\left(\frac{\pi}{2}(s_{1}+s_{2})\right)\cot\left(\frac{\pi s_{2}}{s_{1}+s_{2}}\right)+O\left(\left|s_{1}+s_{2}\right|^{-1}\right).
\]
\end{cor}

\begin{rem}
If we put $\zeta^{1/2}(s_{1},s_{2})\coloneqq\zeta(s_{1},s_{2})+\frac{1}{2}\zeta(s_{1}+s_{2})$,
the equality in the theorem is equivalent to
\[
\frac{\zeta^{1/2}(s_{1},s_{2})}{f(s_{1}+s_{2})}=\frac{1}{2}\cos\left(\frac{\pi}{2}(s_{1}+s_{2})\right)\sum_{j=0}^{2N}\frac{\pi^{j}\cot^{(j)}(\pi r_{2})}{(s_{1}+s_{2})_{j}}c_{j}+O\left(\left|s_{1}+s_{2}\right|^{-N-1}\right).
\]
\end{rem}

\begin{rem}
In \cite[Theorem 6.5]{MMT-behavior}, Matsumoto, Matsusaka, and Tanackov
gave behavior of $\zeta(\overbrace{-k,\dots,-k}^{r})$ when $k$ takes
odd integer values and $k\to\infty$ for any $r$. Especially, when
$r=2$, their result gives an asymptotic behavior 
\begin{equation}
\zeta(-k,-k)\sim\frac{k}{\pi}\left(\frac{k}{2\pi e}\right)^{2k}\qquad(k\to\infty,\ k\in1+2\mathbb{Z}).\label{eq:odd_ind}
\end{equation}
There is no direct connection between Theorem \ref{thm:intro_main}
and (\ref{eq:odd_ind}) since $(s_{1},s_{2})=(-k,-k)$ does not satisfy
the assumption of Theorem \ref{thm:intro_main}.
\end{rem}

\section{Proof}

\subsection{Some lemmas}

By the functional equation for the double zeta function proved by
Matsumoto in \cite{FuncEqDZ}, we have
\begin{align}
 & \frac{1}{(2\pi)^{s_{1}+s_{2}-1}\Gamma(1-s_{1})}\left(\zeta(s_{1},s_{2})-\frac{\Gamma(1-s_{1})\Gamma(s_{1}+s_{2}-1)}{\Gamma(s_{2})}\zeta(s_{1}+s_{2}-1)\right)\nonumber \\
 & =\frac{1}{i^{s_{1}+s_{2}-1}\Gamma(s_{2})}\left(\zeta(1-s_{2},1-s_{1})-\frac{\Gamma(s_{2})\Gamma(1-s_{1}-s_{2})}{\Gamma(1-s_{1})}\zeta(1-s_{1}-s_{2})\right)\label{eq:e1}\\
 & \quad+2i\sin\left(\frac{\pi}{2}(s_{1}+s_{2}-1)\right)F_{+}(s_{1},s_{2}),\nonumber 
\end{align}
where $i^{t}$ is $\exp(\pi it/2)$ for $t\in\mathbb{C}$ and $F_{+}(s_{1},s_{2})$
is an analytic continuation of
\[
\sum_{k=1}^{\infty}\sigma_{s_{1}+s_{2}-1}(k)\Psi(s_{2},s_{1}+s_{2};2\pi ik)
\]
which converges in the region $\Re s_{1}<0$, $\Re s_{2}>1$. Here,
$\sigma_{s_{1}+s_{2}-1}(k)$ and $\Psi(s_{2},s_{1}+s_{2};2\pi ik)$
are defined by
\[
\sigma_{s}(k)=\sum_{m\mid k}m^{s}
\]
and
\[
\Psi(s_{2},s_{1}+s_{2};2\pi ik)=\frac{1}{\Gamma(s_{2})}\int_{0}^{-i\infty}e^{-2\pi iky}y^{s_{2}-1}(1+y)^{s_{1}-1}dy,
\]
respectively.
\begin{lem}
\label{lem:Fplus}If $-\Re s_{1}$ and $-\Re s_{2}$ are large enough,
then we have
\[
F_{+}(s_{1},s_{2})=\frac{-\Gamma(1-s_{2})\zeta(1-s_{2},1-s_{1})}{2\pi ie^{\pi is_{2}}}-\frac{\Gamma(1-s_{2})}{2\pi i}\int_{p+i\infty}^{p-i\infty}\frac{1}{e^{-2\pi iz}-1}z^{s_{2}}\sum_{l=1}^{\infty}(l-z)^{s_{1}-1}\frac{dz}{z},
\]
where $p$ is any real number between $0$ and $1$.
\end{lem}

\begin{proof}
By definition, in the region $\Re s_{1}<0$, $\Re s_{2}>1$, we have
\begin{align*}
 & \sum_{k=1}^{\infty}\sigma_{s_{1}+s_{2}-1}(k)\Psi(s_{2},s_{1}+s_{2};2\pi ik)\\
 & =\sum_{l=1}^{\infty}\sum_{m=1}^{\infty}l^{s_{1}+s_{2}-1}\Psi(s_{2},s_{1}+s_{2};2\pi ilm)\\
 & =\sum_{l=1}^{\infty}\sum_{m=1}^{\infty}l^{s_{1}+s_{2}-1}\frac{1}{\Gamma(s_{2})}\int_{0}^{-i\infty}e^{-2\pi ilmy}y^{s_{2}-1}(1+y)^{s_{1}-1}dy\\
 & =\sum_{l=1}^{\infty}l^{s_{1}+s_{2}-1}\frac{1}{\Gamma(s_{2})}\int_{0}^{-i\infty}\frac{1}{e^{2\pi ily}-1}y^{s_{2}-1}(1+y)^{s_{1}-1}dy\\
 & =\frac{1}{\Gamma(s_{2})}\int_{0}^{-i\infty}\frac{1}{e^{2\pi iz}-1}z^{s_{2}-1}\sum_{l=1}^{\infty}(z+l)^{s_{1}-1}dz\qquad(y=zl^{-1}).
\end{align*}
Thus the meromorphic continuation of $F_{+}(s_{1},s_{2})$ for $\Re s_{1}<0$
and $s_{2}\in\mathbb{C}$ is given by
\begin{align*}
F_{+}(s_{1},s_{2}) & =\frac{1}{(e^{2\pi is_{2}}-1)\Gamma(s_{2})}\int_{C}\frac{1}{e^{2\pi iz}-1}z^{s_{2}-1}\sum_{l=1}^{\infty}(z+l)^{s_{1}-1}dz\\
 & =\frac{\Gamma(1-s_{2})}{2\pi ie^{\pi is_{2}}}\int_{C}\frac{1}{e^{2\pi iz}-1}z^{s_{2}-1}\sum_{l=1}^{\infty}(z+l)^{s_{1}-1}dz,
\end{align*}
where $C$ is the contour which starts from $-i\infty$ and approaches
to the origin and encircles the origin counterclockwisely with a small
radius and back to $-i\infty$. Furthermore, if the real part of $s_{2}$
is large negative enough, then
\[
\int_{C}\frac{1}{e^{2\pi iz}-1}z^{s_{2}-1}\sum_{l=1}^{\infty}(z+l)^{s_{1}-1}dz=G_{+}-G_{-},
\]
where 
\[
G_{\pm}=\int_{\pm p-i\infty}^{\pm p+i\infty}\frac{1}{e^{2\pi iz}-1}z^{s_{2}-1}\sum_{l=1}^{\infty}(z+l)^{s_{1}-1}dz.
\]
Here $p$ is an arbitrary real number between $0$ and $1$, and the
integral paths are along with $\{z\mid\Re(z)=\pm p\}$. Note that
the residue of 
\[
\frac{1}{e^{2\pi iz}-1}z^{s_{2}-1}\sum_{l=1}^{\infty}(z+l)^{s_{1}-1}
\]
at $z=m$ for $m>0$ is given by
\[
\frac{m^{s_{2}-1}\sum_{l=1}^{\infty}(m+l)^{s_{1}-1}}{2\pi i}.
\]
Thus, by the residue theorem,
\[
G_{+}=-\sum_{m>0}m^{s_{2}-1}\sum_{l=1}^{\infty}(m+l)^{s_{1}-1}=-\zeta(1-s_{2},1-s_{1}).
\]
On the other hand,
\begin{align*}
G_{-} & =\int_{-p-i\infty}^{-p+i\infty}\frac{1}{e^{2\pi iz}-1}z^{s_{2}}\sum_{l=1}^{\infty}(z+l)^{s_{1}-1}\frac{dz}{z}\\
 & =e^{\pi is_{2}}\int_{p+i\infty}^{p-i\infty}\frac{1}{e^{-2\pi iz}-1}z^{s_{2}}\sum_{l=1}^{\infty}(l-z)^{s_{1}-1}\frac{dz}{z}\qquad(z\mapsto-z).
\end{align*}
Hence,
\begin{align*}
F_{+}(s_{1},s_{2}) & =\frac{\Gamma(1-s_{2})}{2\pi ie^{\pi is_{2}}}(G_{+}-G_{-})\\
 & =\frac{-\Gamma(1-s_{2})\zeta(1-s_{2},1-s_{1})}{2\pi ie^{\pi is_{2}}}-\frac{\Gamma(1-s_{2})}{2\pi i}\int_{p+i\infty}^{p-i\infty}\frac{1}{e^{-2\pi iz}-1}z^{s_{2}}\sum_{l=1}^{\infty}(l-z)^{s_{1}-1}\frac{dz}{z}.
\end{align*}
\end{proof}
\begin{lem}
\label{lem:Beta_integral}For $s_{1},s_{2}\in\mathbb{C}$ with $\Re(s_{1}+s_{2})<1$,
we have
\[
\int_{p+i\infty}^{p-i\infty}z^{s_{2}-1}(1-z)^{s_{1}-1}dz=\frac{-2\pi i\Gamma(1-s_{1}-s_{2})}{\Gamma(1-s_{1})\Gamma(1-s_{2})},
\]
where $p$ is any real number between $0$ and $1$.
\end{lem}

\begin{proof}
By the identity theorem, it is enough to only consider the case $\Re(s_{1})>0$.
Then, by Cauchy's integral theorem, we have
\begin{align*}
\int_{p+i\infty}^{p-i\infty}z^{s_{2}-1}(1-z)^{s_{1}-1}dz & =\int_{C}z^{s_{2}-1}(1-z)^{s_{1}-1}dz\\
 & =(e^{\pi i(s_{1}-1)}-e^{-\pi i(s_{1}-1)})\int_{1}^{\infty}z^{s_{2}-1}(z-1)^{s_{1}-1}dz,
\end{align*}
where $C$ is the contour which starts from $+\infty$ and approaches
to $1$ and encircles $1$ counterclockwisely with a small radius
and back to $+\infty$. Furthermore, we have
\begin{align*}
 & (e^{\pi i(s_{1}-1)}-e^{-\pi i(s_{1}-1)})\int_{1}^{\infty}z^{s_{2}-1}(z-1)^{s_{1}-1}dz\\
 & =-(e^{\pi is_{1}}-e^{-\pi is_{1}})\int_{0}^{1}(1-x)^{s_{1}-1}x^{-s_{1}-s_{2}}dx\qquad(z=x^{-1})\\
 & =\frac{-(e^{\pi is_{1}}-e^{-\pi is_{1}})\Gamma(s_{1})\Gamma(1-s_{1}-s_{2})}{\Gamma(1-s_{2})}\\
 & =\frac{-2\pi i\Gamma(1-s_{1}-s_{2})}{\Gamma(1-s_{1})\Gamma(1-s_{2})}.\qedhere
\end{align*}
\end{proof}

\subsection{Main theorem}
\begin{thm}
\label{thm:main}Fix a real number $\epsilon>0$ and a nonnegative
integer $N$. Let $s_{1}$ and $s_{2}$ be complex variables satisfying
$\left|\Im s_{1}\right|<\frac{1}{\epsilon}$, $\left|\Im s_{2}\right|<\frac{1}{\epsilon}$,
and $\epsilon<\frac{\Re s_{1}}{\Re(s_{1}+s_{2})}<1-\epsilon$. Then,
as $\Re(s_{1}),\Re(s_{2})\to-\infty$, we have
\begin{align*}
\frac{\zeta(s_{1},s_{2})}{f(s_{1}+s_{2})} & =-\frac{1}{2}\sin\left(\frac{\pi}{2}(s_{1}+s_{2})\right)+\frac{1}{2}\cos\left(\frac{\pi}{2}(s_{1}+s_{2})\right)\sum_{j=0}^{2N}\frac{\pi^{j}\cot^{(j)}(\pi r_{2})}{(s_{1}+s_{2})_{j}}c_{j}\\
 & \quad+\frac{\sin(\pi s_{2})\Gamma(1-s_{1})\Gamma(1-s_{2})\zeta(2-s_{1}-s_{2})}{4\pi\Gamma(1-s_{1}-s_{2})\sin(\frac{\pi}{2}(s_{1}+s_{2}))}+O\left(\left|s_{1}+s_{2}\right|^{-N-1}\right),
\end{align*}
where $r_{j}\coloneqq\frac{s_{j}}{s_{1}+s_{2}}\quad(j=1,2)$ and 
\begin{align*}
c_{j} & \coloneqq{\rm Coeff}((1+xr_{2})^{-s_{1}}(1-xr_{1})^{-s_{2}},x^{j})\\
 & =\sum_{2l_{2}+3l_{3}+\cdots+jl_{j}=j}(s_{1}+s_{2})^{l_{2}+\cdots+l_{j}}\prod_{k=2}^{j}\frac{1}{l_{k}!}\left(\frac{r_{1}(-r_{2})^{k}+r_{2}r_{1}^{k}}{k}\right)^{l_{k}}.
\end{align*}
\end{thm}

\begin{proof}
Without loss of generality, we can assume that $\epsilon$ is small
enough. We assume that the real parts of $s_{1}$ and $s_{2}$ are
large negative enough. By (\ref{eq:e1}) and Lemma \ref{lem:Fplus},
we have
\begin{align*}
\zeta(s_{1},s_{2}) & =\frac{\Gamma(1-s_{1})\Gamma(s_{1}+s_{2}-1)}{\Gamma(s_{2})}\zeta(s_{1}+s_{2}-1)\\
 & \quad+\left(\frac{2\pi}{i}\right)^{s_{1}+s_{2}-1}\frac{\Gamma(1-s_{1})}{\Gamma(s_{2})}\zeta(1-s_{2},1-s_{1})\\
 & \quad-\left(\frac{2\pi}{i}\right)^{s_{1}+s_{2}-1}\Gamma(1-s_{1}-s_{2})\zeta(1-s_{1}-s_{2})\\
 & \quad+(2\pi)^{s_{1}+s_{2}-1}\Gamma(1-s_{1})\cdot\cos\left(\frac{\pi}{2}(s_{1}+s_{2})\right)\cdot\frac{\Gamma(1-s_{2})\zeta(1-s_{2},1-s_{1})}{\pi e^{\pi is_{2}}}\\
 & \quad+(2\pi)^{s_{1}+s_{2}-1}\Gamma(1-s_{1})\cdot\cos\left(\frac{\pi}{2}(s_{1}+s_{2})\right)\cdot\frac{\Gamma(1-s_{2})}{\pi}\int_{p+i\infty}^{p-i\infty}\frac{1}{e^{-2\pi iz}-1}z^{s_{2}}\sum_{l=1}^{\infty}(l-z)^{s_{1}-1}\frac{dz}{z},
\end{align*}
and thus
\begin{align}
\frac{\zeta(s_{1},s_{2})}{f(s_{1}+s_{2})} & =\frac{\Gamma(1-s_{1})\Gamma(s_{1}+s_{2}-1)}{f(s_{1}+s_{2})\Gamma(s_{2})}\zeta(s_{1}+s_{2}-1)\label{eq:e2}\\
 & \quad+\frac{i^{1-s_{1}-s_{2}}\Gamma(1-s_{1})\Gamma(1-s_{2})\sin(\pi s_{2})}{2\pi\Gamma(1-s_{1}-s_{2})}\zeta(1-s_{2},1-s_{1})\nonumber \\
 & \quad-\frac{i^{1-s_{1}-s_{2}}}{2}\zeta(1-s_{1}-s_{2})\nonumber \\
 & \quad+\frac{\Gamma(1-s_{1})\Gamma(1-s_{2})}{2\Gamma(1-s_{1}-s_{2})}\cdot\cos\left(\frac{\pi}{2}(s_{1}+s_{2})\right)\cdot\frac{\zeta(1-s_{2},1-s_{1})}{\pi e^{\pi is_{2}}}\nonumber \\
 & \quad+\frac{\Gamma(1-s_{1})\Gamma(1-s_{2})}{2\pi\Gamma(1-s_{1}-s_{2})}\cdot\cos\left(\frac{\pi}{2}(s_{1}+s_{2})\right)\int_{p+i\infty}^{p-i\infty}\frac{1}{e^{-2\pi iz}-1}z^{s_{2}-1}\sum_{l=2}^{\infty}(l-z)^{s_{1}-1}dz\nonumber \\
 & \quad+\frac{\Gamma(1-s_{1})\Gamma(1-s_{2})}{2\pi\Gamma(1-s_{1}-s_{2})}\cdot\cos\left(\frac{\pi}{2}(s_{1}+s_{2})\right)\int_{p+i\infty}^{p-i\infty}\frac{1}{e^{-2\pi iz}-1}z^{s_{2}}(1-z)^{s_{1}-1}\frac{dz}{z}.\nonumber 
\end{align}
Put $M=\Re(-s_{1}-s_{2})$ and let us consider the asymptotic behavior
of the right hand side of (\ref{eq:e2}) for $M\to\infty$. For the
first term, by the functional equation of the Riemann zeta function,
we have 
\begin{align}
 & \frac{\Gamma(1-s_{1})\Gamma(s_{1}+s_{2}-1)}{f(s_{1}+s_{2})\Gamma(s_{2})}\zeta(s_{1}+s_{2}-1)\nonumber \\
 & =\frac{\Gamma(1-s_{1})}{f(s_{1}+s_{2})\Gamma(s_{2})}\frac{\zeta(2-s_{1}-s_{2})}{2^{2-s_{1}-s_{2}}\pi^{1-s_{1}-s_{2}}\sin(\pi(\frac{2-s_{1}-s_{2}}{2}))}\nonumber \\
 & =\frac{\Gamma(1-s_{1})\zeta(2-s_{1}-s_{2})}{4\Gamma(s_{2})\Gamma(1-s_{1}-s_{2})\sin(\frac{\pi(s_{1}+s_{2})}{2})}\nonumber \\
 & =\frac{\sin(\pi s_{2})\Gamma(1-s_{1})\Gamma(1-s_{2})\zeta(2-s_{1}-s_{2})}{4\pi\Gamma(1-s_{1}-s_{2})\sin(\frac{\pi(s_{1}+s_{2})}{2})}.\label{eq:term_1st}
\end{align}
For the second term, we have
\begin{equation}
\frac{i^{1-s_{1}-s_{2}}\Gamma(1-s_{1})\Gamma(1-s_{2})\sin(\pi s_{2})}{2\pi\Gamma(1-s_{1}-s_{2})}\zeta(1-s_{2},1-s_{1})=O(M^{-h})\label{eq:term_2nd}
\end{equation}
for any $h>0$. For the third term, we have
\begin{equation}
-\frac{i^{1-s_{1}-s_{2}}}{2}\zeta(1-s_{1}-s_{2})=-\frac{i^{1-s_{1}-s_{2}}}{2}+O(M^{-h})\label{eq:term_3rd}
\end{equation}
for any $h>0$. For the fourth term, we have
\begin{equation}
\frac{\Gamma(1-s_{1})\Gamma(1-s_{2})}{2\Gamma(1-s_{1}-s_{2})}\cdot\cos\left(\frac{\pi}{2}(s_{1}+s_{2})\right)\cdot\frac{\zeta(1-s_{2},1-s_{1})}{\pi e^{\pi is_{2}}}=O(M^{-h})\label{eq:term_4th}
\end{equation}
for any $h>0$.

Let us estimate the fifth term. Note that the integral does not depend
on the choice of $0<p<1$.
\begin{align*}
 & \int_{p+i\infty}^{p-i\infty}\frac{1}{e^{-2\pi iz}-1}z^{s_{2}-1}\sum_{l=2}^{\infty}(l-z)^{s_{1}-1}dz\\
 & =\int_{t\in\mathbb{R}}\frac{1}{e^{-2\pi i(p-it)}-1}(p-it)^{s_{2}-1}\sum_{l=2}^{\infty}(l-p+it)^{s_{1}-1}(-i)dt\\
 & =\int_{|t|<2}\frac{1}{e^{-2\pi i(p-it)}-1}(p-it)^{s_{2}-1}\sum_{l=2}^{\infty}(l-p+it)^{s_{1}-1}(-i)dt\\
 & \quad+\int_{2<|t|}\frac{1}{e^{-2\pi i(p-it)}-1}(p-it)^{s_{2}-1}\sum_{l=2}^{\infty}(l-p+it)^{s_{1}-1}(-i)dt.
\end{align*}
Now, let $p=1-\frac{\epsilon}{2}$. Then,
\begin{align*}
 & \int_{|t|<2}\frac{1}{e^{-2\pi i(p-it)}-1}(p-it)^{s_{2}-1}\sum_{l=2}^{\infty}(l-p+it)^{s_{1}-1}(-i)dt\\
 & =O\left(\int_{|t|<2}\left|p-it\right|^{\Re(s_{2})-1}\cdot\sum_{l=2}^{\infty}\left|l-p+it\right|^{\Re(s_{1})-1}dt\right)\\
 & =O\left(p^{\Re(s_{2})-1}\cdot\sum_{l=2}^{\infty}(l-p)^{\Re(s_{1})-1}\right)\\
 & =O(\left|r_{2}^{s_{2}}\right|)\qquad(\left|r_{2}\right|<p).
\end{align*}
On the other hand,
\begin{align*}
 & \int_{2<|t|}\frac{1}{e^{-2\pi i(p-it)}-1}(p-it)^{s_{2}-1}\sum_{l=2}^{\infty}(l-p+it)^{s_{1}-1}(-i)dt\\
 & =O\left(\int_{t=2}^{\infty}\sum_{l=2}^{\infty}\left|l-p+it\right|^{\Re(s_{1})-1}dt\right)\\
 & =O\left(\int_{t=2}^{\infty}t\sum_{l=2}^{\infty}\left((l-p)^{2}+t^{2}\right)^{\frac{\Re(s_{1})-1}{2}}dt\right)\\
 & =O\left(\sum_{l=2}^{\infty}\left[\frac{1}{\Re(s_{1})+1}\left((l-p)^{2}+t^{2}\right)^{\frac{\Re(s_{1})+1}{2}}\right]_{2}^{\infty}\right)\\
 & =O\left(\sum_{l=2}^{\infty}\left((l-p)^{2}+4\right)^{\frac{\Re(s_{1})+1}{2}}\right)\\
 & =O\left(\sum_{l=2}^{\infty}(l-p)^{\Re(s_{1})+1}\right)\\
 & =O(1).
\end{align*}
Thus,
\[
\int_{p+i\infty}^{p-i\infty}\frac{1}{e^{-2\pi iz}-1}z^{s_{2}-1}\sum_{l=2}^{\infty}(l-z)^{s_{1}-1}dz=O(\left|r_{2}^{s_{2}}\right|)+O(1).
\]
Hence
\begin{align}
 & \frac{\Gamma(1-s_{1})\Gamma(1-s_{2})}{2\pi\Gamma(1-s_{1}-s_{2})}\cdot\cos\left(\frac{\pi}{2}(s_{1}+s_{2})\right)\int_{p+i\infty}^{p-i\infty}\frac{1}{e^{-2\pi iz}-1}z^{s_{2}-1}\sum_{l=2}^{\infty}(l-z)^{s_{1}-1}dz\nonumber \\
 & =O(M^{-h})\label{eq:term_5th}
\end{align}
for any $h>0$.

Let us estimate the sixth term. Fix $N>0$. Note that the integral
does not depend on the choice of $0<p<1$. We estimate the integral
by considering the case $p=\Re(r_{2})$. We first decompose as
\begin{align*}
 & \int_{r_{2}+i\infty}^{r_{2}-i\infty}\frac{1}{e^{-2\pi iz}-1}z^{s_{2}}(1-z)^{s_{1}-1}\frac{dz}{z}\\
 & =\int_{r_{2}+i\infty}^{r_{2}-i\infty}z^{s_{2}-1}(1-z)^{s_{1}-1}\sum_{j=0}^{N}a_{j}(z-r_{2})^{j}dz\\
 & \quad+\int_{r_{2}+i\infty}^{r_{2}-i\infty}\left(\frac{1}{e^{-2\pi iz}-1}-\sum_{j=0}^{N}a_{j}(z-r_{2})^{j}\right)z^{s_{2}-1}(1-z)^{s_{1}-1}dz,
\end{align*}
where
\[
a_{j}\coloneqq{\rm Coeff}\left(\frac{1}{e^{-2\pi i(x+r_{2})}-1},x^{j}\right)
\]
is the coefficient of $(z-r_{2})^{j}$ in the Taylor expansion of
$\frac{1}{e^{-2\pi iz}-1}$. Note that $a_{j}$ is bounded for each
$j$ when $\Re(s_{1}),\Re(s_{2})\to-\infty$. Then, for any $\delta>0$,
we have
\begin{align}
 & \int_{r_{2}+i\infty}^{r_{2}-i\infty}\left(\frac{1}{e^{-2\pi iz}-1}-\sum_{j=0}^{N}a_{j}(z-r_{2})^{j}\right)z^{s_{2}-1}(1-z)^{s_{1}-1}dz\nonumber \\
 & =-\int_{t\in\mathbb{R},|t|<M^{-1/2+\delta}}\left(\frac{1}{e^{-2\pi i(r_{2}-it)}-1}-\sum_{j=0}^{N}a_{j}(-it)^{j}\right)(r_{2}-it)^{s_{2}-1}(r_{1}+it)^{s_{1}-1}idt\label{eq:diff_decom}\\
 & \quad-\int_{t\in\mathbb{R},M^{-1/2+\delta}<|t|<2}\left(\frac{1}{e^{-2\pi i(r_{2}-it)}-1}-\sum_{j=0}^{N}a_{j}(-it)^{j}\right)(r_{2}-it)^{s_{2}-1}(r_{1}+it)^{s_{1}-1}idt\nonumber \\
 & \quad-\int_{t\in\mathbb{R},2<|t|}\left(\frac{1}{e^{-2\pi i(r_{2}-it)}-1}-\sum_{j=0}^{N}a_{j}(-it)^{j}\right)(r_{2}-it)^{s_{2}-1}(r_{1}+it)^{s_{1}-1}idt.\nonumber 
\end{align}
Since
\[
\arg(r_{2})=O(M^{-1}),
\]
we have
\[
\left|1\pm\frac{it}{r_{j}}\right|=\left|1\pm e^{\frac{\pi i}{2}-i\arg(r_{j})}\frac{t}{\left|r_{j}\right|}\right|>\left|1-e^{\frac{\pi i}{4}}\frac{t}{\left|r_{j}\right|}\right|\geq\frac{\left|t\right|}{\sqrt{2}\left|r_{j}\right|}\qquad(j=1,2)
\]
when $M$ is large enough. Thus,
\begin{align}
 & \int_{t\in\mathbb{R},2<|t|}\left(\frac{1}{e^{-2\pi i(r_{2}-it)}-1}-\sum_{j=0}^{N}a_{j}(-it)^{j}\right)(r_{2}-it)^{s_{2}-1}(r_{1}+it)^{s_{1}-1}idt\nonumber \\
 & =O\left(\left|r_{1}^{s_{1}}r_{2}^{s_{2}}\right|\int_{t\in\mathbb{R},2<|t|}\left(1+|t|^{N}\right)\left|1-\frac{it}{r_{2}}\right|^{\Re(s_{2})-1}\left|1+\frac{it}{r_{1}}\right|^{\Re(s_{1})-1}dt\right)\nonumber \\
 & =O\left(\left|r_{1}^{s_{1}}r_{2}^{s_{2}}\right|\int_{t\in\mathbb{R},2<|t|}\left|t\right|^{N}\left(\frac{\left|t\right|}{\sqrt{2}\left|r_{2}\right|}\right)^{\Re(s_{2})-1}\left(\frac{\left|t\right|}{\sqrt{2}\left|r_{1}\right|}\right)^{\Re(s_{1})-1}dt\right)\nonumber \\
 & =O\left(\int_{t\in\mathbb{R},2<|t|}\left|t\right|^{N}\left(\frac{\left|t\right|}{\sqrt{2}}\right)^{\Re(s_{1})+\Re(s_{2})-2}dt\right)\nonumber \\
 & =O(M^{-h})\label{eq:diff_1}
\end{align}
for any $h>0$. Furthermore, we have
\begin{align*}
\left|1-\frac{it}{r_{2}}\right| & =\left|1-it\Re\left(\frac{1}{r_{2}}\right)+t\Im\left(\frac{1}{r_{2}}\right)\right|\\
 & =\sqrt{\left(1+t\Im\left(\frac{1}{r_{2}}\right)\right)^{2}+t^{2}\Re\left(\frac{1}{r_{2}}\right)^{2}}\\
 & \geq\sqrt{1+2t\Im\left(\frac{1}{r_{2}}\right)+t^{2}\Re\left(\frac{1}{r_{2}}\right)^{2}}.
\end{align*}
Here, when $\left|t\right|>M^{-1/2+\delta}$ and $M$ is large enough,
we have
\[
\left|2t\Im\left(\frac{1}{r_{2}}\right)\right|\leq\frac{1}{2}t^{2}\Re\left(\frac{1}{r_{2}}\right)^{2}
\]
since
\begin{align*}
\frac{\left|2t\Im\left(\frac{1}{r_{2}}\right)\right|}{\frac{1}{2}t^{2}\Re\left(\frac{1}{r_{2}}\right)^{2}} & =\left|\frac{4}{\Re\left(\frac{1}{r_{2}}\right)^{2}}\times\Im\left(\frac{1}{r_{2}}\right)\times\frac{1}{t}\right|\\
 & =O(1)\times O(M^{-1})\times O(M^{1/2-\delta})\\
 & =O(M^{-\frac{1}{2}-\delta}).
\end{align*}
Thus,
\begin{align*}
\left|1-\frac{it}{r_{2}}\right| & \geq\sqrt{1+\frac{1}{2}t^{2}\Re\left(\frac{1}{r_{2}}\right)^{2}}\\
 & \geq\sqrt{1+\frac{1}{2}M^{-1+2\delta}\Re\left(\frac{1}{r_{2}}\right)^{2}}
\end{align*}
when $\left|t\right|>M^{-1/2+\delta}$ and $M$ is large enough. Similarly,
we also have
\[
\left|1+\frac{it}{r_{1}}\right|\geq\sqrt{1+\frac{1}{2}M^{-1+2\delta}\Re\left(\frac{1}{r_{1}}\right)^{2}}
\]
when $\left|t\right|>M^{-1/2+\delta}$ and $M$ is large enough. Thus

\begin{align}
 & \int_{t\in\mathbb{R},M^{-1/2+\delta}<|t|<2}\left(\frac{1}{e^{-2\pi i(r_{2}-it)}-1}-\sum_{j=0}^{N}a_{j}(-it)^{j}\right)(r_{2}-it)^{s_{2}-1}(r_{1}+it)^{s_{1}-1}idt\nonumber \\
 & =O\left(\left|r_{1}^{s_{1}}r_{2}^{s_{2}}\right|\int_{t\in\mathbb{R},M^{-1/2+\delta}<|t|<2}\left|1-\frac{it}{r_{2}}\right|^{\Re(s_{2})}\left|1+\frac{it}{r_{1}}\right|^{\Re(s_{1})}dt\right)\nonumber \\
 & =O\left(\left|r_{1}^{s_{1}}r_{2}^{s_{2}}\right|\int_{t\in\mathbb{R},M^{-1/2+\delta}<|t|<2}\left(\sqrt{1+\frac{1}{2}M^{-1+2\delta}\Re\left(\frac{1}{r_{2}}\right)^{2}}\right)^{\Re(s_{2})}\left(\sqrt{1+\frac{1}{2}M^{-1+2\delta}\Re\left(\frac{1}{r_{1}}\right)^{2}}\right)^{\Re(s_{1})}dt\right)\nonumber \\
 & =O\left(\left|r_{1}^{s_{1}}r_{2}^{s_{2}}\right|\left(1+cM^{-1+2\delta}\right)^{-M/2}\right)\qquad\left(c\coloneqq\min\left(\frac{1}{2}\Re\left(\frac{1}{r_{2}}\right)^{2},\frac{1}{2}\Re\left(\frac{1}{r_{1}}\right)^{2}\right)\right)\nonumber \\
 & =O\left(\left|r_{1}^{s_{1}}r_{2}^{s_{2}}\right|\left(\left(1+cM^{-1+2\delta}\right)^{\frac{1}{cM^{-1+2\delta}}}\right)^{-cM^{2\delta}/2}\right)\nonumber \\
 & =O\left(\left|r_{1}^{s_{1}}r_{2}^{s_{2}}\right|2^{-cM^{2\delta}/2}\right).\label{eq:diff_2}
\end{align}
Note that
\[
\left|1+\frac{it}{r_{1}}\right|\geq\cos(\arg(r_{1}))=1-O(M^{-2})\qquad\text{and}\qquad\left|1-\frac{it}{r_{2}}\right|\geq\cos(\arg(r_{2}))=1-O(M^{-2})
\]
for all $t\in\mathbb{R}$. Thus, we get
\[
\left|\left(1+\frac{it}{r_{1}}\right)^{s_{1}}\right|=O(1)\qquad\text{and}\qquad\left|\left(1-\frac{it}{r_{2}}\right)^{s_{2}}\right|=O(1).
\]
Therefore, we have
\begin{align}
 & \int_{t\in\mathbb{R},|t|<M^{-1/2+\delta}}\left(\frac{1}{e^{-2\pi i(r_{2}-it)}-1}-\sum_{j=0}^{N}a_{j}(-it)^{j}\right)(r_{2}-it)^{s_{2}-1}(r_{1}+it)^{s_{1}-1}idt\nonumber \\
 & =O\left(\left|r_{1}^{s_{1}}r_{2}^{s_{2}}\right|\int_{t\in\mathbb{R},|t|<M^{-1/2+\delta}}(M^{-1/2+\delta})^{N+1}\left|\left(1+\frac{it}{r_{1}}\right)^{s_{1}}\left(1-\frac{it}{r_{2}}\right)^{s_{2}}\right|dt\right)\nonumber \\
 & =O\left(\left|r_{1}^{s_{1}}r_{2}^{s_{2}}\right|M^{(-1/2+\delta)(N+2)}\right).\label{eq:diff_3}
\end{align}
By (\ref{eq:diff_decom}), (\ref{eq:diff_1}), (\ref{eq:diff_2}),
and (\ref{eq:diff_3}), we have
\begin{align*}
 & \int_{r_{2}+i\infty}^{r_{2}-i\infty}\left(\frac{1}{e^{-2\pi iz}-1}-\sum_{j=0}^{N}a_{j}(z-r_{2})^{j}\right)z^{s_{2}-1}(1-z)^{s_{1}-1}dz\\
 & =O\left(\left|r_{1}^{s_{1}}r_{2}^{s_{2}}\right|M^{(-1/2+\delta)(N+2)}\right)+O\left(\left|r_{1}^{s_{1}}r_{2}^{s_{2}}\right|2^{-cM^{2\delta}/2}\right)+O(M^{-h}),
\end{align*}
where 
\[
c\coloneqq\min\left(\frac{1}{2}\Re\left(\frac{1}{r_{2}}\right)^{2},\frac{1}{2}\Re\left(\frac{1}{r_{1}}\right)^{2}\right),
\]
 and $h$, $\delta$ are any positive real numbers. Thus, by letting
$\delta=\frac{1}{2(N+2)},$ we have
\[
\int_{r_{2}+i\infty}^{r_{2}-i\infty}\left(\frac{1}{e^{-2\pi iz}-1}-\sum_{j=0}^{N}a_{j}(z-r_{2})^{j}\right)z^{s_{2}-1}(1-z)^{s_{1}-1}dz=O\left(\left|r_{1}^{s_{1}}r_{2}^{s_{2}}\right|M^{-(N/2+1/2)}\right).
\]

We have
\begin{align*}
 & \int_{r_{2}+i\infty}^{r_{2}-i\infty}z^{s_{2}-1}(1-z)^{s_{1}-1}\sum_{j=0}^{N}a_{j}(z-r_{2})^{j}dz\\
 & =\int_{r_{2}+i\infty}^{r_{2}-i\infty}z^{s_{2}-1}(1-z)^{s_{1}-1}\sum_{j=0}^{N}a_{j}(zr_{1}-(1-z)r_{2})^{j}dz\\
 & =\sum_{j=0}^{N}a_{j}\sum_{j_{1}+j_{2}=j}{j \choose j_{1}}\int_{r_{2}+i\infty}^{r_{2}-i\infty}z^{s_{2}+j_{2}-1}(1-z)^{s_{1}+j_{1}-1}(-r_{2})^{j_{1}}r_{1}^{j_{2}}dz\\
 & =-2\pi i\sum_{j=0}^{N}a_{j}\sum_{j_{1}+j_{2}=j}{j \choose j_{1}}\frac{\Gamma(1-s_{1}-s_{2}-j)}{\Gamma(1-s_{1}-j_{1})\Gamma(1-s_{2}-j_{2})}(-r_{2})^{j_{1}}r_{1}^{j_{2}}\qquad(\text{by Lemma \ref{lem:Beta_integral}})\\
 & =-2\pi i\sum_{j=0}^{N}a_{j}\frac{\Gamma(1-s_{1}-s_{2})j!}{\Gamma(1-s_{1})\Gamma(1-s_{2})(s_{1}+s_{2})_{j}}\sum_{j_{1}+j_{2}=j}\frac{(s_{1})_{j_{1}}(s_{2})_{j_{2}}}{j_{1}!j_{2}!}(-r_{2})^{j_{1}}r_{1}^{j_{2}}\\
 & =-2\pi i\sum_{j=0}^{N}a_{j}\frac{\Gamma(1-s_{1}-s_{2})j!}{\Gamma(1-s_{1})\Gamma(1-s_{2})(s_{1}+s_{2})_{j}}{\rm Coeff}((1+xr_{2})^{-s_{1}}(1-xr_{1})^{-s_{2}},x^{j})\\
 & =-2\pi i\sum_{j=0}^{N}a_{j}\frac{\Gamma(1-s_{1}-s_{2})j!}{\Gamma(1-s_{1})\Gamma(1-s_{2})(s_{1}+s_{2})_{j}}c_{j}.
\end{align*}
Thus
\begin{align*}
 & \frac{\Gamma(1-s_{1})\Gamma(1-s_{2})}{2\pi\Gamma(1-s_{1}-s_{2})}\cdot\cos\left(\frac{\pi}{2}(s_{1}+s_{2})\right)\int_{p+i\infty}^{p-i\infty}\frac{1}{e^{-2\pi iz}-1}z^{s_{2}}(1-z)^{s_{1}-1}\frac{dz}{z}\\
 & =-i\cos\left(\frac{\pi}{2}(s_{1}+s_{2})\right)\sum_{j=0}^{N}\frac{a_{j}j!}{(s_{1}+s_{2})_{j}}c_{j}+O(M^{-N/2-1/2}).\\
 & =\cos\left(\frac{\pi}{2}(s_{1}+s_{2})\right)\sum_{j=0}^{N}\frac{{\rm Coeff}\left(\frac{-i}{e^{-2\pi i(x+r_{2})}-1},x^{j}\right)j!}{(s_{1}+s_{2})_{j}}c_{j}+O(M^{-N/2-1/2})\\
 & =\cos\left(\frac{\pi}{2}(s_{1}+s_{2})\right)\sum_{j=0}^{N}\frac{{\rm Coeff}\left(\frac{i}{2}+\frac{1}{2}\cot(\pi(x+r_{2})),x^{j}\right)j!}{(s_{1}+s_{2})_{j}}c_{j}+O(M^{-N/2-1/2})\\
 & =\frac{i}{2}\cos\left(\frac{\pi}{2}(s_{1}+s_{2})\right)+\frac{1}{2}\cos\left(\frac{\pi}{2}(s_{1}+s_{2})\right)\sum_{j=0}^{N}\frac{{\rm Coeff}\left(\cot(\pi(x+r_{2})),x^{j}\right)j!}{(s_{1}+s_{2})_{j}}c_{j}+O(M^{-N/2-1/2}).
\end{align*}
Now, by replacing $N$ with $2N+2$, we obtain
\begin{align}
 & \frac{\Gamma(1-s_{1})\Gamma(1-s_{2})}{2\pi\Gamma(1-s_{1}-s_{2})}\cdot\cos\left(\frac{\pi}{2}(s_{1}+s_{2})\right)\int_{p+i\infty}^{p-i\infty}\frac{1}{e^{-2\pi iz}-1}z^{s_{2}}(1-z)^{s_{1}-1}\frac{dz}{z}\nonumber \\
 & =\frac{i}{2}\cos\left(\frac{\pi}{2}(s_{1}+s_{2})\right)+\frac{1}{2}\cos\left(\frac{\pi}{2}(s_{1}+s_{2})\right)\sum_{j=0}^{2N+2}\frac{{\rm Coeff}\left(\cot(\pi(x+r_{2})),x^{j}\right)j!}{(s_{1}+s_{2})_{j}}c_{j}+O(M^{-N-3/2})\nonumber \\
 & =\frac{i}{2}\cos\left(\frac{\pi}{2}(s_{1}+s_{2})\right)+\frac{1}{2}\cos\left(\frac{\pi}{2}(s_{1}+s_{2})\right)\sum_{j=0}^{2N}\frac{{\rm Coeff}\left(\cot(\pi(x+r_{2})),x^{j}\right)j!}{(s_{1}+s_{2})_{j}}c_{j}+O(M^{-N-1}).\label{eq:term_6th}
\end{align}
Thus, by (\ref{eq:e2}), (\ref{eq:term_1st}), (\ref{eq:term_2nd}),
(\ref{eq:term_3rd}), (\ref{eq:term_4th}), (\ref{eq:term_5th}),
and (\ref{eq:term_6th}), we have
\begin{align*}
\frac{\zeta(s_{1},s_{2})}{f(s_{1}+s_{2})} & =\frac{\sin(\pi s_{2})\Gamma(1-s_{1})\Gamma(1-s_{2})\zeta(2-s_{1}-s_{2})}{4\pi\Gamma(1-s_{1}-s_{2})\sin(\frac{\pi}{2}(s_{1}+s_{2}))}\\
 & \quad-\frac{i^{1-s_{1}-s_{2}}}{2}+\frac{i}{2}\cos\left(\frac{\pi}{2}(s_{1}+s_{2})\right)\\
 & \quad+\frac{1}{2}\cos\left(\frac{\pi}{2}(s_{1}+s_{2})\right)\sum_{j=0}^{2N}\frac{{\rm Coeff}\left(\cot(\pi(x+r_{2})),x^{j}\right)j!}{(s_{1}+s_{2})_{j}}c_{j}+O(M^{-N-1})\\
 & =\frac{\sin(\pi s_{2})\Gamma(1-s_{1})\Gamma(1-s_{2})\zeta(2-s_{1}-s_{2})}{4\pi\Gamma(1-s_{1}-s_{2})\sin(\frac{\pi}{2}(s_{1}+s_{2}))}\\
 & \quad-\frac{1}{2}\sin\left(\frac{\pi}{2}(s_{1}+s_{2})\right)\\
 & \quad+\frac{1}{2}\cos\left(\frac{\pi}{2}(s_{1}+s_{2})\right)\sum_{j=0}^{2N}\frac{\pi^{j}\cot^{(j)}(\pi r_{2})}{(s_{1}+s_{2})_{j}}c_{j}+O(M^{-N-1}).
\end{align*}
Finally, the explicit formula for $c_{j}$ follows from
\begin{align*}
(1+xr_{2})^{-s_{1}}(1-xr_{1})^{-s_{2}} & =\exp(-s_{1}\log(1+xr_{2})-s_{2}\log(1-xr_{1}))\\
 & =\exp\left(s_{1}\sum_{k=1}^{\infty}\frac{(-xr_{2})^{k}}{k}+s_{2}\sum_{k=1}^{\infty}\frac{(xr_{1})^{k}}{k}\right)\\
 & =\exp\left((s_{1}+s_{2})\sum_{k=2}^{\infty}\frac{r_{1}(-r_{2})^{k}+r_{2}r_{1}^{k}}{k}x^{k}\right)\\
 & =\prod_{k=2}^{\infty}\sum_{l=0}^{\infty}\frac{(s_{1}+s_{2})^{l}}{l!}\left(\frac{r_{1}(-r_{2})^{k}+r_{2}r_{1}^{k}}{k}\right)^{l}x^{kl}.\qedhere
\end{align*}
\end{proof}
Now Theorem \ref{thm:intro_main} follows from Theorem \ref{thm:main}
because the additional condition
\[
\min\{\left|s_{1}+s_{2}-2k\right|:k\in\mathbb{Z}\}>\frac{1}{|s_{1}+s_{2}|^{1+1/\epsilon}}
\]
in Theorem \ref{thm:intro_main} implies that
\[
\frac{\sin(\pi s_{2})\Gamma(1-s_{1})\Gamma(1-s_{2})\zeta(2-s_{1}-s_{2})}{4\pi\Gamma(1-s_{1}-s_{2})\sin(\frac{\pi}{2}(s_{1}+s_{2}))}=O\left(\left|s_{1}+s_{2}\right|^{-h}\right)
\]
for any $h>0$.

\subsection*{Acknowledgements}

This work was supported by JSPS KAKENHI Grant Numbers JP18K13392,
JP19K14511, JP22K03244, and JP22K13897. This work was also supported
by ``Grant for Basic Science Research Projects from The Sumitomo Foundation''
and by ``Research Funding Granted by The University of Kitakyushu''.

\end{document}